\documentclass[a4paper,11pt]{amsart}
\usepackage{amsmath,amssymb,amsthm,enumerate,fullpage}

\usepackage{lmodern}

\theoremstyle{plain} 
\newtheorem{theorem}{Theorem}[section]
\newtheorem{cor}[theorem]{Corollary}
\newtheorem{lemma}[theorem]{Lemma}
\newtheorem{prop}[theorem]{Proposition}

\theoremstyle{definition}

\newtheorem{example}[theorem]{Example}

\theoremstyle{remark}

\numberwithin{equation}{section}

\DeclareMathOperator{\im}{Im}

\DeclareMathOperator{\rad}{rad}
\DeclareMathOperator{\soc}{soc}

\newcommand{\ra}{\rightarrow}
\newcommand{\field}[1]{\mathbb{#1}}

\newcommand{\F}{\ensuremath{\field{F}}}
\newcommand{\C}{\ensuremath{\field{C}}}

\newcommand{\Q}{\ensuremath{\field{Q}}}

\newcommand{\Z}{\ensuremath{\field{Z}}}

\newcommand{\size}[1]{\lvert #1 \rvert}
\newcommand{\abs}[1]{\lvert #1 \rvert}

\newcommand{\U}{\ensuremath{U}}

\newcommand{\A}{\ensuremath{A}}

\newcommand{\ve}{\varepsilon}
\newcommand{\Glt}{G_{\mathrm{lt}}}
\newcommand{\Grt}{G_{\mathrm{rt}}}

\newcommand{\zint}{\mathbb Z}

\newcommand{\To}{\longrightarrow}
\newcommand{\upontop}[2]{\genfrac{}{}{0pt}{}{#1}{#2}}

\begin{document}

\title[Extension Theorem]{MacWilliams' Extension Theorem\\ 
  for Bi-Invariant Weights over\\
  Finite Principal Ideal Rings}

\date{\today}

\author[Greferath et al.]{Marcus Greferath}
\author[]{Thomas Honold}
\author[]{Cathy Mc Fadden}
\author[]{Jay A. Wood}
\author[]{Jens Zumbr\"agel}


\address{%
  University College Dublin\\
  Zhejiang University \\
  Western Michigan University}
\email{\sloppy \{marcus.greferath, cathy.mcfadden, 
  jens.zumbragel\}@ucd.ie, honold@zju.edu.cn, jay.wood@wmich.edu}

\thanks{This work was partially supported by Science Foundation
  Ireland under Grants 06/MI/006 and 08/IN.1/I1950, and by a
  sabbatical leave from Western Michigan University.  JAW thanks the
  Claude Shannon Institute for its hospitality and support during a
  research visit in November 2011.}


\dedicatory{In memoriam Werner Heise (1944--2013)}


\begin{abstract}
  A finite ring $R$ and a weight $w$ on $R$ satisfy the
  \emph{Extension Property} if every $R$-linear $w$-isometry between
  two $R$-linear codes in $R^n$ extends to a monomial transformation
  of $R^n$ that preserves $w$.  MacWilliams proved that finite fields
  with the Hamming weight satisfy the Extension Property.  It is known
  that finite Frobenius rings with either the Hamming weight or the
  homogeneous weight satisfy the Extension Property.  Conversely, if a
  finite ring with the Hamming or homogeneous weight satisfies the
  Extension Property, then the ring is Frobenius.

  This paper addresses the question of a characterization of all
  bi-invariant weights on a finite ring that satisfy the Extension
  Property. Having solved this question in previous papers for all
  direct products of finite chain rings and for matrix rings, we have
  now arrived at a characterization of these weights for finite
  principal ideal rings, which form a large subclass of the finite
  Frobenius rings. We do not assume commutativity of the rings in
  question.

  \bigskip

  \noindent {\sc Key Words:} Frobenius ring, principal ideal ring,
  linear code, extension theorem, M\"obius function.

  \medskip

  \noindent $2010$ Mathematics Subject Classification: Primary 94B05;
  Secondary 16S36, 20M25.

\end{abstract}

\maketitle


\section{Introduction}

Let $R$ be a finite ring equipped with a weight $w$.  Two linear codes
$C, D \le {_RR^n}$ are \emph{isometrically equivalent} if there is an
isometry between them, i.e., an $R$-linear bijection $\varphi: C
\longrightarrow D$ that satisfies $w(\varphi(c))=w(c)$ for all $c\in
C$. We say that $\varphi$ \emph{preserves} the weight $w$.

MacWilliams in her doctoral dissertation \cite{macw63} and later
Bogart, Goldberg, and Gordon~\cite{bogagoldgord78} proved that, in the
case where $R$ is a finite field and $w$ is the Hamming weight, every
isometry is the restriction of a {\em monomial transformation} $\Phi$
of the ambient space $_RR^n$.  A monomial transformation of $_RR^n$ is
simply a left linear mapping $\Phi: R^n\longrightarrow R^n$ the matrix
representation of which is a product of a permutation matrix and an
invertible diagonal matrix.  Said another way, every Hamming isometry
over a finite field extends to a monomial transformation.  This result
is often called the \emph{MacWilliams Extension Theorem} or the
\emph{MacWilliams Equivalence Theorem}.

With increased interest in linear codes over finite rings there arose
the natural question: could the Extension Theorem be proved in the
context of ring-linear coding theory? This question appeared
complicated, as two different weights were pertinent: the traditional
Hamming weight $w_{\rm H}$ and also a new weight $w_{\rm hom}$ called
the \emph{homogeneous} weight by its discoverers Constantinescu
and Heise~\cite{consheis97}.

In~\cite{wood99} Wood proved the MacWilliams Extension Theorem
for all linear codes over finite Frobenius rings equipped with the
Hamming weight. In the commutative case he showed in the same paper
that the Frobenius property was not only sufficient but also
necessary.  In the non-commutative case, the necessity of the
Frobenius property was proved in \cite{wood08a}.

Inspired by the paper of Constantinescu, Heise, and
Honold~\cite{consheishono96} which used combinatorial methods to prove
the Extension Theorem for homogeneous weights on $\zint_m$, Greferath
and Schmidt~\cite{grefschm00} showed that the Extension Theorem is
true for linear codes over finite Frobenius rings when using the
homogeneous weight.  Moreover, they showed that for all finite rings
every Hamming isometry between two linear codes is a homogeneous
isometry and vice versa.

The situation can be viewed as follows: for $R$ a finite ring, and
either the Hamming weight or the homogeneous weight, the Extension
Theorem holds for all linear codes in $R^n$ if and only if the ring is
Frobenius.  This is a special case of more general results by
Greferath, Nechaev, and Wisbauer~\cite{grefnechwisb04} who proved that
if the codes are submodules of a quasi-Frobenius bi-module $_RA_R$
over any finite ring $R$, then the Extension Theorem holds for the
Hamming and homogeneous weights.  The converse of this was proved by
Wood in \cite{wood09}. \medskip

Having understood all requirements on the algebraic side of the
problem, we now focus on the metrical aspect. This paper aims to
further develop a characterization of all weights on a finite
(Frobenius) ring, for which the corresponding isometries satisfy the
Extension Theorem.

In our discussion we will assume that the weights in question are
bi-invariant, which means that $w(ux) = w(x) =w(xu)$ for all $x\in R$
and $u\in R^\times$.  Our main results do not apply to weights with
smaller symmetry groups such as the Lee or Euclidean weight (on
$R=\Z_m$, except for $m\in\{2,3,4,6\}$), despite their importance for
ring-linear coding theory.

The goal of this paper is to give a necessary and sufficient condition
that a bi-invariant weight $w$ must satisfy in order for the Extension
Theorem to hold for isometries preserving $w$.  We are not able to
characterize all such weights when the underlying ring is an arbitrary
Frobenius ring, but we do achieve a complete result for
\emph{principal ideal rings}. These are rings in which each left or
right ideal is principal, and they form a large subclass of the finite
Frobenius rings.

The present work is a continuation and generalization of earlier work
on this topic \cite{wood97, wood99a, grefhono05, grefhono06, wood09,
  grefmcfazumb13}.  As in \cite{grefhono06, grefmcfazumb13} the
M\"obius function on the partially ordered set of (principal, right)
ideals is crucial for the statement and proof of our main
characterization result; however, in contrast to these works we do not
need the values of the M\"obius function explicitly, but use its
defining properties instead to achieve a more general result.  Our
restriction to principal ideal rings stems from our method of proof,
which requires the annihilator of a principal ideal to be principal.
The main result was proved for the case of finite chain rings in
\cite[Theorem~3.2]{grefhono05} (and in a more general form in
\cite[Theorem~16]{wood97}), in the case $\Z_m$ in
\cite[Theorem~8]{grefhono06}, for direct products of finite chain
rings in \cite[Theorem~22]{grefmcfazumb13}, and for matrix rings over
finite fields in \cite[Theorem~9.5]{wood09} (see
Example~\ref{ex:examples} below).  The main result gives a concrete
manifestation of \cite[Proposition~12]{wood97} and
\cite[Theorem~3.1]{wood99a}.  Further to \cite{grefmcfazumb13} we
prove that our condition on the weight is not only sufficient, but
also necessary for the Extension Theorem, using an argument similar to
that in \cite{grefhono06, wood08a}. \medskip

Here is a short summary of the contents of the paper.  In
Section~\ref{sec:prelims} we review the terminology of Frobenius
rings, M\"obius functions, and orthogonality matrices needed for the
statements and proofs of our main results.  In addition, we prove a
result (Corollary~\ref{cor_smult}) that says that a right-invariant
weight $w$ on $R$ satisfies the Extension Property if the Hamming
weight $w_{\rm H}$ is a correlation multiple of $w$.

In Section~\ref{sec:orthog-matrices} we show that the Extension
Property holds for a bi-invariant weight if and only if its
orthogonality matrix is invertible.  The main results are stated in
Section~\ref{sec:bi-inv-wts}.  By an appropriate unimodular change of
basis, the orthogonality matrix can be put into triangular form, with
a simple expression for the diagonal entries (Theorem~\ref{thm:WQ}).
The Main Result (Theorem~\ref{maintheorem}) then says that the
Extension Property holds if and only if all the diagonal entries of
the orthogonality matrix are nonzero.  A proof of Theorem~\ref{thm:WQ}
is given in Section~\ref{sec:proof}. \medskip

This paper is written in memory of our friend, teacher, and colleague
Werner Heise who, sadly, passed away in February 2013 after a long
illness. Werner has been very influential in ring-linear coding theory
through his discovery of the homogeneous weight on ${\mathbb Z}_m$
(``Heise weight'') and subsequent contributions.


\section{Notation and Background}%
\label{sec:prelims}

In all that follows, rings $R$ will be finite, associative and possess
an identity $1$. The group of invertible elements (units) will be
denoted by $R^\times$ or $\U$.  Any module $_RM$ will be unital,
meaning $1m=m$ for all $m\in M$.

\subsection*{Frobenius Rings}

We describe properties of Frobenius rings needed in this paper, as in
\cite{honold01}.

The character group of the additive group of a ring $R$ is defined as
$\widehat{R}:={\rm Hom}_{\mathbb Z}(R,{\mathbb C}^\times)$.
This group has the structure of an $R,R$-bimodule by defining
$\chi^r(x):=\chi(rx)$ and $^r\chi(x):=\chi(xr)$ for all $r,x\in R$,
and for all $\chi\in \widehat{R}$.

The \emph{left socle} ${\rm soc}(_RR)$ is defined as the sum of all
minimal left ideals of $R$. It is a two-sided ideal.  A similar
definition leads to the \emph{right socle} ${\rm soc}(R_R)$ which is
also two-sided, but will not necessarily coincide with its left
counterpart.

A finite ring $R$ is  
\emph{Frobenius} if one of the following four
equivalent statements holds:
\begin{itemize}\itemsep=1mm
\item $_RR \cong {_R\widehat{R}}$.
\item $R_R \cong {\widehat{R}_R}$.
\item ${\rm soc}(_RR)$ is left principal.
\item ${\rm soc}(R_R)$ is right principal.
\end{itemize} 
For a finite Frobenius ring the left and right socles coincide.

Crucial for later use is the fact that finite Frobenius rings are
quasi-Frobenius and hence possess a perfect duality. This means the
following: Let $L(_RR)$ denote the lattice of all left ideals of $R$,
and let $L(R_R)$ denote the lattice of all right ideals of $R$. There
is a mapping $\perp: L(_RR) \longrightarrow L(R_R),\; I \mapsto
I^\perp$ where $I^\perp:= \{x\in R \mid Ix=0\}$ is the right
annihilator of $I$ in $R$. This mapping is an order anti-isomorphism
between the two lattices.  The inverse mapping associates to every
right ideal its left annihilator.

\subsection*{Principal Ideal Rings}

A ring $R$ is \emph{left principal} if every left ideal is left
principal, similarly a ring is \emph{right principal} if every right
ideal is right principal.  If a ring is both left principal and right
principal it is a {\em principal ideal ring}.  Nechaev in
\cite{nechaev73} proved that ``a finite ring with identity in which
every two-sided ideal is left principal is a principal ideal ring."
Hence every finite left principal ideal ring is a principal ideal
ring.  Further, as argued in \cite{nechaev73}, the finite principal
ideal rings are precisely the finite direct sums of matrix rings over
finite chain rings.  They form a subclass of the class of finite
Frobenius rings (since, for example, their one-sided socles are
principal).

\subsection*{M\"obius Function}

The reader who is interested in a more detailed survey of the
following is referred to \cite[Chapter~IV]{aigner}, \cite{rota64}, or
\cite[Chapter~3.6]{stanley}.

For a finite partially-ordered set (poset) $P$, we have the incidence
algebra
\[ {\mathbb A}(P) \;:=\; \{\,f: P\times P \longrightarrow {\mathbb Q} 
\mid\, x \not\le y \;\; \mbox{implies} \;\; f(x,y)=0 \,\} \:. \] 
Addition and scalar multiplication in ${\mathbb A}( P)$ are defined
point-wise; multiplication is convolution:
\[  (f*g)(a,b) = \sum_{a \le c \le b} f(a,c) \, g(c,b) \:. \]
The invertible elements are exactly the functions $f\in {\mathbb
  A}(P)$ satisfying $f(x,x) \ne 0$ for all $x\in P$.  In particular,
the characteristic function of the partial order of $P$ given by
\[ \zeta: P\times P \longrightarrow {\mathbb Q} \:,
\quad (x,y) \mapsto \left\{\begin{array}{lcl}
    1 & : & x\le y\\
    0 & : & \mbox{otherwise}
  \end{array}\right. \] 
is an invertible element of ${\mathbb A}(P)$.  Its inverse is the {\em
  M\"obius function\/} $\mu: P\times P \longrightarrow {\mathbb Q}$
implicitly defined by $\mu(x,x) = 1$ and \[ \sum_{x\le t \le y}
\mu(x,t) \;=\; 0 \] if $x<y$, and $\mu(x,y) = 0$ if $x \not\le y$.

\subsection*{Weights and Code Isometries}

Let $R$ be any finite ring.  By a \emph{weight} $w$ we mean any
$\Q$-valued function $w: R \longrightarrow {\mathbb Q}$ on $R$,
without presuming any particular properties.  As usual we extend $w$
additively to a weight on $R^n$ by setting
\[ w: R^n \longrightarrow {\mathbb Q} \:,\quad 
x \mapsto \sum_{i=1}^n w(x_i) \:. \] 
The \emph{left} and \emph{right symmetry groups} of $w$ are defined by
\[ \Glt(w) := \{ u \in U: w(ux) = w(x), x \in R\} \:, \quad 
\Grt(w) := \{ v \in U: w(xv) = w(x), x \in R\} \:. \]
A weight $w$ is called \emph{left} (resp.\ \emph{right})
\emph{invariant} if $\Glt(w) = U$ (resp.\ $\Grt(w) = U$).

A (left) \emph{linear code} of length $n$ over $R$ is a submodule $C$
of $ {}_{R}R^{n}$. A {\em $w$-isometry\/} is a linear map $\varphi : C
\To {}_{R}R^{n}$ with $w(\varphi(x)) = w(x)$ for all $x \in C$, i.e.,
a mapping that preserves the weight $w$.

A \emph{monomial transformation} is a bijective (left) $R$-linear
mapping $\Phi : R^{n} \longrightarrow R^{n}$ such that there is a
permutation $\pi\in S_n$ and units $u_1, \ldots, u_n\in \U$ so that \[
\Phi(x_1, \ldots, x_n) \;=\; ( x_{\pi(1)} u_1 , \dots , x_{\pi(n)}
u_n ) \] for every $(x_1 , \dots , x_n ) \in R^{n}$. In other words,
the matrix that represents $\Phi$ with respect to the standard basis
of $_RR^n$ decomposes as a product of a permutation matrix and an
invertible diagonal matrix.  A \emph{$\Grt(w)$-monomial
  transformation} is one where the units $u_i$ belong to the right
symmetry group $\Grt(w)$.  A $\Grt(w)$-monomial transformation is a
$w$-isometry of $R^n$.

We say that a finite ring $R$ and a weight $w$ on $R$ satisfy the
\emph{Extension Property} if the following holds: For every positive
length $n$ and for every linear code $C\le {_RR^n}$, every injective
$w$-isometry $\varphi : C \longrightarrow {_RR}^{n}$ is the
restriction of a $\Grt(w)$-monomial transformation of $_RR^{n}$.  That
is, every injective $w$-isometry $\varphi$ extends to a monomial
transformation that is itself a $w$-isometry of $R^n$. \medskip

Let $w: R \longrightarrow {\mathbb Q}$ be a weight and let $f: R
\longrightarrow {\mathbb Q}$ be any function. We define a new weight
$wf$ as
\[ wf: R \longrightarrow {\mathbb Q} \:, \quad 
x \mapsto \sum_{r\in R} w(rx)\,f(r) \:. \] 
By the operation of \emph{right correlation} $(w,f)\mapsto wf$, the
vector space $V := \Q^R$ of all weights on $R$ becomes a right module
$V_\A$ over $\A = \Q[(R,\cdot)]$, the rational semigroup algebra of
the multiplicative semigroup $(R,\cdot)$ of the ring
(see~\cite{grefmcfazumb13}).  For $r\in R$ denote by $e_r$ the weight
where $e_r(r) = 1$ and $e_r(s) = 0$ for $s\ne r$.  Then $we_r$ is
simply given by $(we_r)(x) = w(rx)$.

Denote the natural additive extension of $wf$ to $R^n$ by $wf$ also.

\begin{lemma}\label{lem_smult}
  Let $C \le {_RR^n}$ be a linear code and let $\varphi: C
  \longrightarrow R^n$ be a $w$-isometry, then $\varphi$ is also a
  $wf$-isometry for any function $f: R \longrightarrow {\mathbb Q}$.
\end{lemma}
  
\begin{proof}
  For all $x\in C$ we compute 
  \begin{align*}
    (wf) (\varphi(x))
    & \;=\; \sum_{r\in R} w(r\varphi(x)) \, f(r) \;=\; 
    \sum_{r\in R} w(\varphi(rx)) \, f(r)  \\
    & \;=\;  \sum_{r\in R} w(rx) \, f(r) 
    \;=\; (wf)(x) \:. \qedhere
  \end{align*}  
\end{proof}

For a weight $w$ consider the $\Q$-linear map $\tilde w: \A \to V$,
$f\mapsto wf$.  By Lemma~\ref{lem_smult}, if $\varphi$ is a
$w$-isometry then $\varphi$ is a $w'$- isometry for all $w'
\in\im\tilde w$.  Note that $\im\tilde w = w\A \le V_\A$.

\subsection*{Weights on Frobenius Rings}

Now let $R$ be a finite Frobenius ring.  We describe two approaches
that ultimately lead to the same criterion for a weight $w$ to satisfy
the Extension Property.

\subsubsection*{Approach~1} From earlier work~\cite{wood99} we know
that the Hamming weight $w_{\rm H}$ satisfies the Extension Property.
Combining this fact with Lemma~\ref{lem_smult}, we immediately obtain
the following result.

\begin{cor}\label{cor_smult}
  Let $R$ be a finite Frobenius ring and let $w$ be a weight on $R$
  such that $\Grt(w) = U$ and $wf = w_{\rm H}$ for some function
  $f:R\to\Q$.  Then $w$ satisfies the Extension Property.
\end{cor}

In other words, if $w$ is right-invariant and $w_{\rm H}\in\im\tilde
w$ then $w$ satisfies the Extension Property.

How can we make sure that $w_{\rm H}\in\im\tilde w$?  One idea is to
show that the $\Q$-linear map $\tilde w$ is bijective: Using the
natural basis $(e_r)_{r\in R}$ for $V$ and the property $(we_r)(s) =
w(rs)$ it is easy to see that $\tilde w$ is described by the transpose
of the matrix $(w(rs))_{r,s\in R}$.  However, if the weight function
$w$ is left- or right-invariant {\em or} satisfies $w(0) = 0$ then
this matrix is not invertible.  Therefore we work with a ``reduced''
version of the map $\tilde w$.

As before, let $V := \Q^R$ be the vector space of all weights on $R$,
and let $V_0^\U$ be the subspace of all weights $w$ satisfying $w(0) =
0$ that are right-invariant.  Similarly, we define the subspace $^\U
V_0$ of all weights $w$ with $w(0) = 0$ that are left-invariant.  The
corresponding invariant subspaces of $\A = \Q[(R,\cdot)]$ are
$\A_0^\U$ and $^\U \A_0$, where $\A_0 := \A / \Q e_0$.

If $w$ is a weight in $V_0^\U$ then $wf\in V_0^\U$ for {\em any}
function $f:R\to\Q$, i.e., $\im\tilde w \le V_0^\U$.  In this case we
could examine the bijectivity of the $\Q$-linear map $\tilde w:
\A_0^\U \to V_0^\U$ (the restriction of the above map $\tilde w$).
But this map does not have a nice matrix representation; setting
$e_{s\U} = \sum_{r\in s\U}e_r$ and letting $(e_{s\U})_{s\U\ne 0}$ be
the natural basis for $\A_0^\U$ and for $V_0^U$, the entries of the
matrix turn out to be sums of several values $w(rus)$.

However, if we work with the restriction $\tilde w: {}^\U \A_0 \to
V_0^\U$ instead and if the weight $w$ is bi-invariant (i.e., both
left- and right-invariant), then, with respect to the natural bases,
this $\Q$-linear map does have a nice matrix description, namely the
orthogonality matrix.  This will be explained below. If this map
$\tilde w$ is invertible, then $w$ satisfies the Extension Property by
Corollary~\ref{cor_smult}.

Note: Since $\im\tilde w$ is a submodule of $V_\A$ it follows that
$w_{\rm H}\in\im\tilde w$ if and only if $\im\tilde w_{\rm H} \le \im\tilde w$.
Actually, $\im\tilde w_{\rm H} = V_0^\U$ (see
Proposition~\ref{prop_om-reverse} below), so that $w_{\rm
  H}\in\im\tilde w$ if and only if $V_0^U \subseteq \im\tilde w$.
This is why it is a sensible approach to investigate the
surjectivity/bijectivity of the map $\tilde w$.

\subsubsection*{Approach~2} The same orthogonality matrix that appears
in Approach~1 also appears in \cite{wood97}.  By
\cite[Proposition~12]{wood97} (also, \cite[Theorem~3.1]{wood99a} and
\cite[Section~9.2]{wood09}), the invertibility of the orthogonality
matrix of $w$ implies that a $w$-isometry preserves the so-called
\emph{symmetrized weight composition} associated with $\Grt(w)$.
Then, \cite[Theorem~10]{wood97} shows that any injective linear
homomorphism that preserves the symmetrized weight composition
associated with $\Grt(w)$ extends to a $\Grt(w)$-monomial
transformation.  Thus, if the orthogonality matrix is invertible, any
$w$-isometry extends to a $\Grt(w)$-monomial transformation, and hence
$w$ satisfies the Extension Property.

\subsection*{Orthogonality Matrices}

Let $R$ be a finite Frobenius ring. There is a one-to-one
correspondence between left (resp., right) principal ideals and left
(resp., right) $\U$-orbits.  Each $\U$-orbit is identified with the
principal ideal of which its elements are the generators
(\cite[Proposition~5.1]{wood99}, based on work of Bass).  Define for
$r, s\in R\setminus \{0\}$ the functions $\ve_{R r} (x) = \size{\U
  r}^{-1} $ if $x \in \U r$, i.e., if $Rr = Rx$, and zero otherwise;
similarly, let $e_{sR} (x) = e_{sU}(x) = 1$ if $xR = sR$ and zero
otherwise.  Then $(\ve_{R r})$ and $(e_{sR})$ are bases for $^\U \A_0$
and $V_0^\U$, as $R r$ and $sR$ vary over all left and right nonzero
principal ideals of $R$, respectively.
 
For a bi-invariant weight $w$, define the \emph{orthogonality matrix}
of $w$ by $W_0 = \big(w(rs) \big){}_{Rr\ne 0,\,sR\ne 0}$.  That is,
the entry in the $Rr, sR$-position is the value of the weight $w$ on
the product $rs \in R$.  The value $w(rs)$ is well-defined, because
$w$ is bi-invariant.  Note that $W_0$ is square; this follows from
work of Greferath \cite{gref02} that shows the equality of the number
of left and right principal ideals in a finite Frobenius ring.

\begin{prop}\label{prop_matrix}
  Suppose $w$ is bi-invariant with $w(0)=0$.  Then 
  \[ w \, \ve_{R r} = \sum_{sR\ne 0} w(rs) \, e_{s R} \] 
  for nonzero $R r$, where the sum extends over all the nonzero right
  principal ideals $s R$.  In particular, the matrix representing the
  $\Q$-linear map $\tilde w: {}^\U \A_0 \to V_0^\U$, $f\mapsto wf$,
  with respect to the bases $(\ve_{R r})$ and $(e_{sR})$, is the
  transpose of the matrix $W_0$.
\end{prop}

\begin{proof}
  Since $w\in V_0^\U$ we have $w \, \ve_{R r} \in V_0^\U$, and therefore
  \[  w \, \ve_{R r} = \sum_{sR\ne 0} (w \, \ve_{R r})(s) \, e_{s R} \:.  \]
  Calculating, using that $w\in {}^\U V_0$, we get:
  \[ (w \, \ve_{R r})(s) = \sum_{t \in R} w(ts) \, \ve_{R r}(t) 
    = \sum_{t \in Ur} \size{\U r}^{-1} w(ts)
    = w(rs) \:. \qedhere \]
\end{proof}

In the algebraic viewpoint of \cite{grefmcfazumb13}, $V_0^{\U}$ is a
right module over $^\U \!\A_0$.  Then, $W_0$ is invertible if and only
if $w$ is a generator for $V_0^{\U}$.

If $R$ is a finite field and $w = w_{\rm H}$, the Hamming weight on
$R$, then $W_0$ is exactly the orthogonality matrix considered by
Bogart, Goldberg, and Gordon~\cite[Section~2]{bogagoldgord78}.  More
general versions of the matrix $W_0$ have been utilized in
\cite{wood97, wood99a, wood09}.

\begin{example} 
  For $R={\mathbb Z}_4$ the Lee weight $w_{\rm Lee}$ assigns $0
  \mapsto 0$, $1\mapsto 1$, $2\mapsto 2$ and $3\mapsto 1$. It is a
  bi-invariant weight function, as is the Hamming weight $w_{\rm H}$
  on $R$. Based on the natural ordering of the (nonzero) principal
  ideals of $R$ as $2R < R$ the orthogonality matrix for $w_{\rm Lee}$
  is
  \[ W_0^{\rm Lee} \, = \, \left[ \begin {array}{cc}
      0 & 2 \\
      2 & 1
    \end{array}\right], \] 
  whereas the orthogonality matrix for $w_{\rm H}$ is given by 
  \[ W_0^{\rm H} \, = \, \left[ \begin {array}{cc}
      0 & 1\\
      1 & 1
    \end{array}\right]. \]
  Both of these matrices are invertible over ${\mathbb Q}$ as observed
  in \cite{gref02}, where it was shown that the Extension Property is
  satisfied.
\end{example}


\section{Orthogonality Matrices and the Extension Theorem}%
\label{sec:orthog-matrices}

In the present section we will show that invertibility of the
orthogonality matrix of a bi-invariant weight is necessary and
sufficient for that weight to satisfy the Extension Property.  We
split this result into two statements.

\begin{prop} 
  Let $R$ be a finite Frobenius ring and let $w$ be a bi-invariant
  weight on $R$.  If the orthogonality matrix $W_0$ of $w$ is
  invertible, then $w$ satisfies the Extension Property.
\end{prop}

\begin{proof}
  Approach~1: by Proposition~\ref{prop_matrix} the matrix $W_0$
  describes the $\Q$-linear map $\tilde w: {}^\U \A_0 \to V_0^\U$,
  $f\mapsto wf$.  Hence if $W_0$ is invertible the map $\tilde w$ is
  bijective, and in particular $w_{\rm H}\in\im\tilde w$.  Thus by
  Corollary~\ref{cor_smult} the weight $w$ satisfies the Extension
  Property.
  
  Approach~2: apply \cite[Proposition~12]{wood97} or
  \cite[Theorem~3.1]{wood99a}.
\end{proof}

We remark that in the foregoing discussion, $\Q$ could be replaced
throughout by any field $K$ containing $\Q$, for example $K = \C$.

\begin{prop}\label{prop_om-reverse}
  Let $R$ be a finite Frobenius ring, and let $w$ be a bi-invariant
  rational weight on $R$ that satisfies the Extension Property.  Then
  the orthogonality matrix $W_0$ of $w$ is invertible.
\end{prop}

\begin{proof}
  The proof mimics that of \cite[Theorem~4.1]{wood08a} and
  \cite[Proposition~7]{grefhono06}.  Assume $W_0$ singular for the
  sake of contradiction. Then there exists a nonzero rational vector
  $v = (v_{cR})_{cR\ne 0}$ such that $W_0 v = 0$.  Without loss of
  generality, we may assume that $v$ has integer entries.  We proceed
  to build two linear codes $C_+, C_-$ over $R$.  Each of the codes
  will have only one generator.  The generator for $C_{\pm}$ is a
  vector $g_{\pm}$ with the following property: for each ideal $cR \le
  R_R$ with $v_{cR} > 0$ (for $g_+$), resp., $v_{cR} < 0$ (for $g_-$),
  the vector $g_{\pm}$ contains $\abs{v_{cR}}$ entries equal to
  $c$. To make these two generators annihilator-free, we append to
  both a trailing $1\in R$. The typical codeword in $C_{\pm}$ is hence
  of the form $a g_{\pm}$ for suitable $a \in R$.  We compare $w(a
  g_+)$ and $w(a g_-)$ for every $a \in R$ by calculating the
  difference $D( a ) = w(a g_+) - w(a g_-)$.  By our construction of
  the generators $g_{\pm}$, we have \[ D( a ) \;=\; \sum_{cR\ne 0}
  w(ac) \, v_{cR} \;=\; (W_0 v)_{Ra} \;=\; 0 \:, \] for all $a\in R$.
  Thus $a g_+ \mapsto a g_-$ forms a $w$-isometry from $C_+$ to
  $C_-$. The codes, however, are not monomially equivalent because
  their entries come from different right $\U$-orbits.
\end{proof}


We summarize our findings in the following theorem.

\begin{theorem}
  A rational bi-invariant weight function on a finite Frobenius ring
  satisfies the Extension Property if and only if its orthogonality
  matrix is invertible.
\end{theorem}

The ultimate goal is to give necessary and sufficient conditions on a
bi-invariant weight $w$ on a finite Frobenius ring $R$ so that its
orthogonality matrix $W_0$ is invertible.  We are able to derive such
a result for finite principal ideal rings.

\subsection*{Extended Orthogonality Matrices}

Let $R$ be a finite Frobenius ring and let $w$ be a bi-invariant
weight function with $w(0) = 0$.  The orthogonality matrix for the
weight $w$ was defined as $W_0 = \big( w(rs) \big){}_{Rr\ne 0,\,sR\ne
  0}$.  Now define the {\em extended} orthogonality matrix for $w$ as
$W = \big( w(rs) \big){}_{Rr,\,sR}$.  In order to examine the
invertibility of $W_0$ we obtain a formula for $\det W$, the
determinant of the matrix $W$.  (Note that $\det W$ is well-defined up
to multiplication by $\pm1 $, the sign depending on the particular
orderings of the rows and columns of $W$.)  First we relate $\det W$
to $\det W_0$, viewing $w(0)$ as an indeterminate.

\begin{prop}\label{prop_W-W0}
  The determinant $\det W_0$ is obtained from $\det W$ by dividing
  $\det W$ by $w(0)$ and then setting $w(0) = 0$.
\end{prop}

\begin{proof}
  We treat $w(0)$ as an indeterminate $w_0$.  Up to a sign change in
  $\det(W)$, we may assume that the rows and columns of $W$ are
  arranged so that the first row is indexed by $R0$ and the first
  column is indexed by $0R$.  Then $W$ has the form
  \[ W = \left[ \begin{array}{c|c}
      w_0 & w_0 \,\cdots\, w_0 \\ \hline
      w_0 & \\
      \vdots & W' \\
      w_0 & 
    \end{array} \right] . \]
  By subtracting the first row from every other row, we find that
  $\det W = w_0 \det(W'-w_0J)$, where $J$ is the all-one matrix.
  Finally the matrix $W_0$ equals the matrix $W'-w_0J$ evaluated at
  $w_0 = 0$, so that $\det W_0 = \det(W' - w_0J)|_{w_0 = 0}$.
\end{proof}

Note that the extended orthogonality matrix $W$ is not invertible for
weights $w$ satisfying $w(0) = 0$.


\section{Bi-invariant Weights with Invertible Orthogonality\\ 
  Matrix on Principal Ideal Rings}%
\label{sec:bi-inv-wts}

Let $R$ be a finite principal ideal ring, and let $w$ be a
bi-invariant weight on $R$.  Assume $W$ is the extended orthogonality
matrix of $w$.  We are interested in the determinant of $W$ and look
for a way to evaluate this determinant.

We will define an invertible matrix $(Q_{cR,Rx})_{cR,\, Rx}$ with
determinant $\pm 1$ and multiply $W$ by $Q$ from the right to arrive
at $WQ$; then $\det(WQ) = \pm\,\det(W)$.  The most significant
advantage of considering $WQ$, rather than $W$, is that $WQ$ will be a
lower triangular matrix for which we can easily calculate the
determinant.

Define for any finite ring the matrix $Q$ by
\[ Q_{cR, Rx} \;:=\; \mu( (Rx)^\perp, cR ) \: , \]
for $cR\le R_R$ and $Rx\le {_RR}$, where $\mu$ is the M\"obius
function of the lattice $L^*$ of all right ideals of $R$.
    
\begin{lemma}\label{lemQinvertible}
  For a finite principal ideal ring $R$, the matrix $Q$ is an
  invertible matrix with determinant $\pm1$.
\end{lemma}

\begin{proof}
  We claim that the inverse of $Q$ is given by $T_{Ra, bR} :=
  \zeta(bR, (Ra)^\perp)$, where $\zeta$ is the indicator function of
  the poset $L^*$, meaning
  \[ \zeta(xR, yR) \;=\; \left\{\begin{array}{ccl}
      1 & : & xR \le yR \:, \\
      0 & : & \mbox{otherwise} \:.
    \end{array}\right.\]
  We compute the product $TQ$,
  \[ (TQ)_{Ra,Rx} \;=\; \sum_{cR} {  \zeta(cR, (Ra)^\perp)
     \, \mu ( (Rx)^\perp , cR)} \:. \] 
    By the definition of $\zeta$ and the fact that
  $\mu((Rx)^\perp,cR) =0$ unless $(Rx)^\perp \le cR$, the expression
  above simplifies to
  \[ (TQ)_{Ra,Rx} \;= \sum_{ (Rx)^\perp \le cR \le (Ra)^\perp} {\mu
    ((Rx)^\perp , cR)} \:, \] which is $1$ for $(Rx)^\perp =
  (Ra)^\perp$ and $0$ otherwise by the definition of the M\"obius
  function.
  
  The matrix $T$ is upper triangular with $1$s on the main diagonal.
  Thus $\det T$ and hence $\det Q$ equal $\pm 1$.  (The $\pm 1$ allows
  for different orders of rows and columns.)
\end{proof}

\begin{example}
  Let $R := \F_q[x,y] / \langle x^2, y^2 \rangle$, which is a
  commutative local Frobenius ring.  (When $q=2^k$, $R$ is isomorphic
  to the group algebra over $\F_{2^k}$ of the Klein $4$-group.)  Here,
  $(Rxy)^\perp = xR + yR$ is not principal and thus the above proof
  does not apply; in fact, the matrix $Q$ turns out to be singular in
  this case.

  On the other hand, the Frobenius ring $R$ is not a counter-example
  to the main result below.  In fact, $\det(W_0) = \pm q \,
  w(xy)^{q+3}$ satisfies the formula in
  \eqref{eq:det-W0-factorization} below (up to a nonzero
  multiplicative constant), so that the main result still holds over
  $R$.
\end{example}

We are now ready to state the main theorems.  The proof of the next
result is contained in the final section.

\begin{theorem}\label{thm:WQ}
  If $R$ is a finite principal ideal ring, then the matrix $WQ$ is
  lower triangular.  The diagonal entry at position $(Ra, Ra)$ is
  $\sum\limits_{dR \le aR} w( d ) \, \mu(0, dR)$.
\end{theorem}

We conclude that the determinant of $WQ$ and hence that of $W$ is
given by
\[ \det(W) \;=\; \pm\, \det(WQ) \;=\; 
\pm\prod_{aR} \, \sum_{dR\le aR} w(d)\, \mu(0,dR) \:. \]
Applying Proposition~\ref{prop_W-W0} we find the determinant of $W_0$
to be 
\begin{equation}
 \label{eq:det-W0-factorization}
 \det(W_0) \;=\; \pm\prod_{aR\ne 0} \, \sum_{0\ne dR\le
  aR} w(d)\, \mu(0,dR) \:, 
\end{equation}
as in $\det(W)$ the term $aR = 0$ provides a factor of $w(0)$ which
gets divided away, and in each remaining term the contribution from
$dR =0R$ is $w(0)$ which is set equal to $0$.

This yields our main result: a characterization of all bi-invariant
weights on a principal ideal ring that satisfy the Extension Property.

\begin{theorem}[Main Result]\label{maintheorem}
  Let $R$ be a finite principal ideal ring and let $\mu$ be the
  M\"obius function of the lattice $L^*$ of all right ideals of $R$.
  Then a bi-invariant rational weight $w$ on $R$ satisfies the
  Extension Property if and only if
  \[ \sum_{0\ne dR\le aR} w(d)\, \mu(0,dR) \;\ne\, 0\quad \mbox{for all
    $aR \ne 0$} \:. \]
\end{theorem}

The condition in Theorem~\ref{maintheorem} needs to be checked only
for nonzero right ideals $aR\leq\soc(R_R)$, since we have
$\mu(0,dR)=0$ if $dR\not\leq\soc(R_R)$ (see
\cite[Proposition~2]{st:homippi}, for example) and since every right
ideal contained in $\soc(R_R)$ is principal.  As a consequence, the
Extension Property of $w$ depends only on the values of $w$ on the
socle of $R$.

\begin{example} 
  For a chain ring $R$, the main result simply says that a
  bi-invariant weight function $w$ satisfies the Extension Property if
  and only if it does not vanish on the socle of $R$ (compare with
  \cite{grefhono05} and \cite[Theorem~9.4]{wood09}).  For $R =
  {\mathbb Z}_4$, it states that a bi-invariant weight will satisfy
  the Extension Property if and only if $w(2) \ne 0$.
\end{example}
   
\begin{example}  
  Let $R := {\mathbb Z}_m$.  The nonzero ideals in $\soc(\Z_m)$ are of
  the form $a\Z_m$ with $a\mid m$ and $m/a>1$ square-free.  The
  M\"obius function of such an ideal is
  $\mu(0,a\Z_m)=\mu(m/a)=(-1)^r$, where $\mu(\cdot)$ denotes the
  one-variable M\"obius function of elementary number theory and $r$
  is the number of different prime divisors of $m/a$.  According to
  Theorem~\ref{maintheorem}, an invariant weight $w$ on $\Z_m$ has the
  Extension Property if and only if
  \[ \sum_{s\mid\frac{m}{a}} w(sa) \, \mu 
  \Big( \frac{m}{sa} \Big) \;=\; (-1)^r \,
  \sum_{s\mid\frac{m}{a}} w(sa) \, \mu(s) \;\ne\; 0\] 
  for all (positive) divisors $a$ of $m$ such that $m/a$ is
  square-free and $>1$.  We thus recover the main theorem of
  \cite{grefhono06}.
\end{example}
  
\begin{example}\label{ex:examples} 
  Let $R := {\rm Mat}_n(\F_q)$, $n\geq2$, the ring of
  $n\times n$ matrices over the finite field $\F_q$ with $q$ elements,
  so that $\U = {\rm GL}_n(\F_q)$.  The ring $R$ is a finite principal
  ideal ring that is not a direct product of chain rings.  For each
  matrix $A\in R$, the left $\U$-orbit $\U A$ can be identified with
  the row space of $A$, and similarly, the right $\U$-orbits
  correspond to the column spaces.

  Let $w$ be a bi-invariant weight on $R$.  Its value $w(A)$ depends
  only on the rank of the matrix $A$, and therefore we can write
  $w([{\rm rank}\,A]) := w(A)$.  Now for $n=2$, the main result says
  that $w$ satisfies the Extension Property if and only if $w([1])\ne
  0$ and $q\,w([2]) \ne (q+1)\,w([1])$.  For $n=3$, $w$ satisfies the
  Extension Property if and only if $w([1])\ne 0$, $\,q\,w([2]) \ne
  (q+1)\,w([1])$, and $q^3\,w([3]) + (q^2+q+1)\,w([1]) \ne
  (q^2+q+1)\,q\,w([2])$.
    
  It was shown in \cite[Theorem~9.5]{wood09} that the relevant 
  non-vanishing sums are
  \begin{equation}\label{eq:wmatrixterm}
    \sum_{i=1}^{s}{(-1)^i q^{(\upontop{i}{2})} 
      \left[\upontop{s}{i}\right]_q w([i]) }  \:,
  \end{equation}
  where $[\upontop{s}{i}]_q$ is the $q$-nomial (Cauchy binomial)
  defined as
  \[ \left[\upontop{k}{l}\right]_q \,:=\;
  \frac{(1 - q^k)(1-q^{k-1}) \dots (1-q^{k-l+1})}
  {(1-q^l)(1-q^{l-1}) \dots (1-q)}  \:. \]
  
  The {\em rank metric} $w([k]) := k$ satisfies these conditions.
  First we state the Cauchy binomial theorem:
  \begin{equation*}\label{eq:cauchythm}
    \prod_{i=0}^{k-1}{(1 + xq^i)} \;=\; 
    \sum_{j=0}^{k}{\left[\upontop{k}{j}\right]_q  q^{(\upontop{j}{2})} x^j} \:.
  \end{equation*}
  Now we write the term in \eqref{eq:wmatrixterm} for the rank metric,
  changing the sign and including $i=0$ trivially in the sum. This can
  then be seen as the evaluation of a derivative.
  \[ \sum_{i=0}^{s}{i(-1)^{i-1}  \ q^{(\upontop{i}{2})} 
    \left[\upontop{s}{i}\right]_q } \;=\;
  \left. \frac{d}{dx}\sum_{i=0}^{s}{x^i q^{(\upontop{i}{2})} 
      \left[\upontop{s}{i}\right]_q  }\right|_{x=-1} \:. \]
  Applying the Cauchy binomial theorem and evaluating the derivative yields: 
  \[ \left. \frac{d }{dx} \prod_{i=0}^{s-1}{(1+xq^i)} \right|_{x=-1} 
  \;=\; \left( \prod_{i=0}^{s-1}{(1-q^i)} \right) 
  \left( \sum_{i=0}^{s-1}{\frac{q^i}{1-q^i}} \right) \:. \]
  Both expressions on the right are nonzero provided $q$ is not $\pm
  1$, independent of $s$.  Hence the rank metric satisfies the
  Extension Property for all $q$ and $n$.
\end{example} 
    
\begin{example}
  More generally, let $R = {\rm Mat}_n(S)$ be a matrix ring over a
  finite chain ring $S$.  Then $\soc(R_R) = \soc({}_RR) = {\rm
    Mat}_{n\times n}(\soc S) \cong {\rm Mat}_{n\times n}(\F_q)$ as a
  (bi-)module over the residue class field $S/\rad S\cong\F_q$.  Hence
  the previous example applies and characterizes all bi-invariant
  weights $w\colon R\to\Q$ having the Extension Property.
\end{example}

\begin{example}  
  Any finite semisimple ring is a direct product of matrix rings over
  finite fields and therefore a principal ideal ring.  Hence, the main
  result also applies to this case.
\end{example}


\section{A Proof of Theorem~\ref{thm:WQ}}%
\label{sec:proof}

We perform the matrix multiplication and see that the entry of $WQ$ in
position $(Ra,Rb)$ is given by the expression
\[ (WQ)_{Ra,Rb} \;=\; \sum_{cR} W_{Ra,cR} \, Q_{cR,Rb} \;=\;
\sum_{cR} w(ac)\, \mu((Rb)^\perp ,cR) \:. \] 
According to the definition of the M\"obius function, $\mu((Rb)^\perp,
cR)$ can be nonzero only when $(Rb)^\perp \le cR$ (or: when $cR \in
[(Rb)^\perp, R]$, using interval notation on the lattice $L^*$ of all
right (necessarily principal) ideals of $R$).  With this in mind we
write
\begin{equation}\label{main}
  (WQ)_{Ra,Rb} \;= \sum_{cR \in [(Rb)^\perp, R]} w(ac) \, \mu((Rb)^\perp ,cR) \:.
\end{equation}

\subsection*{Diagonal Entries}

The diagonal terms of $WQ$ are given by 
\[ (WQ)_{Ra,Ra} \;= \sum_{ cR \in [(Ra)^\perp,R]} w(ac)\, \mu((Ra)^\perp , cR) \:. \]

For an element $a\in R$ consider the left multiplication operator
$L_a: R \longrightarrow aR,\; t\mapsto at$.  The mapping $L_a$ is a
(right) $R$-linear mapping with kernel $(Ra)^\perp$, and the
isomorphism theorem yields an induced order isomorphism of intervals
\[ \nu_a:[(Ra)^\perp , R] \longrightarrow [0,aR] \:, \quad
J \mapsto aJ \:. \] 
It follows that if $J_1, J_2\in [(Ra)^\perp,R]$, then $\mu(J_1,J_2) =
\mu ( \nu_a(J_1), \nu_a(J_2) ) = \mu (aJ_1, aJ_2)$.

The diagonal term simplifies to
\begin{align*}
  (WQ)_{Ra,Ra} &\;=\;
  \sum_{cR \in [(Ra)^\perp,R]} w(ac)\, \mu((Ra)^\perp , cR) \\
  &\;=\; \sum_{acR \in [0,a R]} w(ac)\, \mu(0, acR) \\
  &\;=\; \sum_{dR \in [0,a R]} w(d)\, \mu(0, dR) \:,
\end{align*}
where we have applied the above interval isomorphism with $J_1 =
(Ra)^\perp$ and $J_2 = cR$, followed by the relabeling $acR=dR$.

Finally, observe that the formula $(WQ)_{Ra,Ra} = \sum_{dR \in [0,a
  R]} w(d)\, \mu(0, dR)$ does not depend on the choice of generator
$a$ for the left ideal $Ra$.  Indeed, any other generator has the form
$ua$, where $u$ is a unit of $R$.  Left multiplication by $u$ induces
an order isomorphism of intervals $\nu_u: [0,aR] \longrightarrow
[0,uaR]$, so that $\mu(0,dR) = \mu(0,udR)$ for all $dR \in [0,aR]$.
Since $w$ is left-invariant, we have $w(ud) = w(d)$, and the right
side of the formula is well-defined.

\subsection*{Lower Triangularity}

Now let us return to the general form of the matrix $WQ$ given in
\eqref{main}.  We would like to prove that $WQ$ is lower triangular,
i.e., that $Rb \nleq Ra$ will imply that $(WQ)_{Ra,Rb}=0$. To that
end, assume
\begin{equation} \label{eq:lower-tri}
  Rb \nleq Ra \:.
\end{equation}
As above, the left multiplication operator $L_a$ induces a mapping
$\lambda_a : [0,R] \ra [0,aR]$, which in turn induces a partition on
$[0,R]$ in a natural way. We first rewrite the general expression for
$(WQ)_{Ra,Rb}$ taking into account this partition.

\[ (WQ)_{Ra,Rb} \;= \sum_{dR \in [0, aR]} w(d) 
\mathop{\sum_{cR \in [(Rb)^\perp,R] }}_{\lambda_a(cR)=dR} \mu((Rb)^\perp ,cR) \:. \]

Our goal is to examine the inner sum and show that it vanishes for
every $dR$ in question. In other words, we will show that \[
\mathop{\sum_{cR \in [(Rb)^\perp,R] }}_{\lambda_a(cR)=dR}
\mu((Rb)^\perp ,cR) \;=\; 0 \:,\quad \mbox{for all $dR \le aR$} \:. \]

We do this by induction on $dR$ in the partially ordered set $[0,aR]$.
Accordingly, we assume the existence of some $dR\in [0,aR]$ which is
minimal with respect to the property that
\[ \mathop{\sum_{cR \in [(Rb)^\perp,R] }}_{\lambda_a(cR)=dR} \mu((Rb)^\perp ,cR) 
\;\ne\; 0 \:. \] 

Consider the right ideal $K := L_a^{-1}(dR) = \sum\limits_{acR \le dR}
cR$. For this ideal we have $(Ra)^\perp \le K$, and moreover, $cR \le
K$ is equivalent to $acR \le dR$. For this reason
\[ \mathop{\sum_{cR \in [(Rb)^\perp,R] }}_{acR\le dR} \mu((Rb)^\perp ,cR)
\;=\; \sum_{ cR \in [(Rb)^\perp , K]} \mu((Rb)^\perp ,cR) \:. \]
By properties of $\mu$, the latter expression is nonzero if and only
if $K=(Rb)^\perp$. This would however imply $(Rb)^\perp \ge
(Ra)^\perp$ (because $(Ra)^\perp \le K$) and hence $Rb\le Ra$,
contrary to assumption \eqref{eq:lower-tri}.  Hence, we conclude
\[ 0 \;= \mathop{\sum_{cR \in [(Rb)^\perp,R]}}_{acR\le dR} \mu((Rb)^\perp ,cR) 
\;= \mathop{\sum_{cR \in [(Rb)^\perp,R]}}_{acR=dR} \mu((Rb)^\perp ,cR) + 
\mathop{\sum_{cR \in [(Rb)^\perp,R]}}_{acR<dR} \mu((Rb)^\perp ,cR) \:.  \]

In this equation the minimality property of $dR$ implies that the last
term vanishes.  This finally forces
\[ \mathop{\sum_{cR \in [(Rb)^\perp,R]}}_{acR=dR} \mu((Rb)^\perp ,cR) \;=\; 0 \:, \] 
contradicting the minimality property of $dR$.  Lower triangularity
follows and this finishes the proof of Theorem~\ref{thm:WQ}.\qed

Note that this proof heavily relies on the hypothesis that $R$ is a
finite principal ideal ring. For a general finite Frobenius ring the
architecture of a proof will need to be vastly restructured.
Nonetheless, we conjecture that the main result, as stated, holds over
any finite Frobenius ring.

\bibliographystyle{amsplain}

\bibliography{paper}

\end{document}